\documentclass{article}
\usepackage{graphicx}
\usepackage{amsmath}
\usepackage{url}
\usepackage{amsfonts}
\usepackage{amssymb}
\usepackage{proof}
\usepackage[all]{xy}
\newtheorem{theorem}{Theorem}

\newtheorem{corollary}[theorem]{Corollary}

\newtheorem{definition}[theorem]{Definition}

\newtheorem{lemma}[theorem]{Lemma}

\newtheorem{remark}[theorem]{Remark}

\newenvironment{proof}[1][Proof]{\textbf{#1.} }{\ \rule{0.5em}{0.5em}}

\begin{document}

%modcor = Maietti:ModCor
%aritj = Maietti:AritJ
\title{Subspaces of an arithmetic universe via type theory}
\author{Maria Emilia Maietti\\
Dipartimento di Matematica Pura ed Applicata, University of Padova,\\
via Trieste n.63, 35121, Padova, Italy\\
maietti@math.unipd.it\\
%% Steven Vickers\\
%% School of Computer Science, University of Birmingham,\\
%% Birmingham, B15 2TT, UK\\
%% s.j.vickers@cs.bham.ac.uk
}
\maketitle
\begin{abstract}
We define the notion of subspace of an arithmetic
universe by using  its internal dependent type theory.
\end{abstract}

\section{Introduction\label{IntroSectn}}
In the recent submitted paper  with Steve Vickers~\cite{mv:arithind}
 we defined the notion of subspace
of an arithmetic universe as a free categorical structure built by means of the partial logic
in  \cite{PHLCC}.

Here we show how we can define subspaces of an arithmetic universe
by using its internal type theory in 
\cite{Maietti:ModCor}.

In the following we use the abbreviation AU for ``arithmetic universe''
as defined in \cite{Maietti:AritJ}. There we gave a general notion
of the instance of arithmetic universes built by Andr\'e Joyal~\cite{Joyal}
in the seventies.
By an AU functor between arithmetic universes we mean a functor preserving
the AU structure up to isomorphisms.

By a subspace of $\cal A$ 
  we mean an AU with extra structure $S$, expressed in terms of new arrows and commutativities, to be added to $\cal A$ and we call it ${\cal A}[S]_t$ 
(where $t$ stands for a type-theoretic description of the free structure).
We take as its universal property the following one:
we have an AU embedding functor $\mathcal{I}:\mathcal{A}\rightarrow \mathcal{A}[S]_t$ and for any AU $\mathcal{B}$,
the category of AU functors
 $\mathbf{AU}(\mathcal{A}[S]_t,\mathcal{B})$ is equivalent to the
category of pairs $(F,\alpha)$ where $F:\mathcal{A}\rightarrow \mathcal{B}$ is an AU functor and $\alpha$ interprets the structure in $S$ with respect to $F$
as in \cite{Maietti:ModCor}.

To show the existence of such subspaces we first define
the internal language of an arithmetic universe $\cal A$ as the free
arithmetic universe generated from $\cal A$, as defined in 
\cite{Maietti:ModCor},
 to which we add coherent isomorphisms
making the free AU-structure added to  $\cal A$ isomorphic to the existing
AU structure in $\cal A$. We call $T_{iso}({\cal A})$ the internal type theory of $\cal A$ 
with such coherent isomorphisms.  This internal type theory of   $\cal A$
differs from that 
defined in \cite{Maietti:ModCor}, and called $T({\cal A})$,
because the first has coherent isomorphisms.
This difference becomes clear when we look at the embedding of  $\cal A$
in the syntactic categories $ \mathcal{C}_{ T(\mathcal{A})}$ and $\mathcal{C}_{ T_{iso}(\mathcal{A})}$ built out of   $T({\cal A})$ and $T_{iso}({\cal A})$
respectively: while 
the category $\cal A$ embeds into $ \mathcal{C}_{ T(\mathcal{A})}$ via a functor
preserving the AU structure strictly, it embeds in
 $\mathcal{C}_{ T_{iso}(\mathcal{A})}$ only via an AU functor.

Then we show that the category of AU functors from the AU-category  $\cal A$
to the AU $\cal B$ 
 is in equivalence with that of
translations from the internal type theory with coherent isomorphisms
of $\cal A$ to that of $\cal B$.

Then we can define a subspace $\mathcal{A}[S]_t$ of an arithmetic universe $\cal A$
with extra structure $S$, expressed in terms of new arrows and commutativities
between them, as the syntactic category of the extension of 
$ T_{iso}(\mathcal{A})$
with the extra structure. In this way we can prove the desired universal
property that  AU functors from $\mathcal{A}$ in an AU $\cal B$ with
the necessary structure to interpret the extra structure $S$ lift to  AU functors
from $\mathcal{A}[S]_t$ in an uniquely up to iso way.

If we define subspaces with extra structure by using 
$ T(\mathcal{A})$ instead of $ T_{iso}(\mathcal{A})$ we just get
a subspace satisfying a lifting property only for functors preserving the
AU structure strictly. 

%%%%%%%%%%%%%%%%%%%%%%%%%%%%%%%%%%%%%%%%%%%%%%%%%%%%%%%%%%%%
\section{Arithmetic universes}

Arithmetic universes are very much the creation of Andr\'{e} Joyal, in
unpublished work from the 1970s. The general notion was not clearly defined,
and we shall follow \cite{Maietti:AritJ} (which also discusses their
background in some detail) in defining them as \emph{list arithmetic pretoposes.}

\begin{definition}\label{AUDefn}
A \emph{pretopos} is a category equipped with finite limits, stable finite
disjoint coproducts and stable effective quotients of equivalence relations.
(For more detailed discussion, see, e.g., \cite[A1.4.8]{Elephant1}.)

A finitely complete category has \emph{parameterized list objects} (see
\cite{Maietti:AritJ}; also \cite{Cockett:Locoi}) if for any object $A$ there
is an object $\mathsf{List}(A)$ with maps $\mathsf{r}_{0}^{A}:1\rightarrow
\mathsf{List}(A)$ and $\mathsf{r}_{1}^{A}:\mathsf{List}(A)\times
A\rightarrow\mathsf{List}(A)$ such that for every $b:B\rightarrow Y$ and
$g:Y\times A\rightarrow Y$ there is a unique $\mathsf{rec}(b,g)$ making the
following diagrams commute%
$$
\def\objectstyle{\scriptstyle}
\def\labelstyle{\scriptstyle}
{\xymatrix @-0.1pc{
B\ar[rr]^{<id, r_{o}^{A}\cdot !_{B}>} \ar[drr]_{b}& &  B\times {\tt List}(A)
\ar[d]^{rec_{l}(b,g)}&&
 B\times ({\tt List}(A)\times A)\ar[ll]_{id\times r_{1}^{A}\quad}\ar[d]^{(rec_{l}(b,g) \times id_{A})\cdot \alpha}\\
& &Y && Y\times A \ar[ll]^{g}
}}
$$
where $\alpha:B\times(\mathsf{List}(A)\times A)\rightarrow(B\times
\mathsf{List}(A))\times A$ is the associativity isomorphism.

An \emph{arithmetic universe} (or \emph{AU}) \cite{Maietti:AritJ} is a
pretopos with parameterized list objects.
We assume that each arithmetic universe is equipped with a {\it choice}
of its structure. For example,
given two objects $A,B$ we can choose their product
and the pairing morphisms of two morphisms.
Note that an AU has all coequalizers, not just the quotients of equivalence
 relations as shown in \cite{Maietti:AritJ}.

This is because the list objects allow one to construct the transitive
closure of any relation.

A functor between AUs is an \emph{AU functor} if it preserves the AU structure
(finite limits, finite colimits, list objects) non-strictly, i.e. up to
isomorphism. We write $\mathbf{AU}$ for the category of AUs and AU functors.
(We shall sometimes refer to a strict AU functor, preserving structure on the nose,
as an \emph{AU homomorphism}.)
\end{definition}

%% We shall be interested in understanding $\mathbf{AS}$ as a ``category of
%% (generalized) spaces'', and in particular we shall be interested in
%% understanding a slice $\mathbf{AS}/X$ as a category of spaces ``fibred over
%% $X$''. Typical questions, for a map $f:Y\rightarrow X$, would be -- When is
%% $f$ fibrewise discrete (i.e. a local homeomorphism or sheaf, hence
%% corresponding to an object of $\mathcal{A}X$)? When is it a subspace
%% inclusion? When is it an open or closed subspace inclusion? And in all those
%% we should like to know something about the structure of $\mathcal{A}Y$ and the
%% AU functor $f^{\ast}$. This paper takes a step towards addressing those questions.

%% In the first question, on fibrewise discreteness, one expects the fibrewise
%% discrete spaces over $X$ to be equivalent to objects of $\mathcal{A}X$ and
%% hence form an AU. Taylor \cite{Taylor:InASDisAU} has investigated the
%% analogous question in his system of Abstract Stone Duality and shown that
%% there too the discrete spaces (or, to be precise in his terminology, the overt
%% discrete spaces) form an AU.

\subsection{Free structures via type theory}
In order to adjoin structure freely
to an AU   we can use its internal type theory devised in
\cite{Maietti:ModCor}.

We start by recalling the necessary notions from \cite{Maietti:ModCor}.
\begin{definition}[${\cal T}_{au}$-theory]
\label{AUExtnDefn}
We write ${\cal T}_{au}$ for the typed calculus that provides the internal language
of arithmetic universes in \cite[section 3]{Maietti:ModCor}.

We call a theory $T$ of the typed calculus
of arithmetic universes ${\cal T}_{au}$,
(in short: a \emph{$\mathcal{T}_{au}$-theory}),
a typed calculus extended with judgements of the form

$$B\ [\Gamma ]\qquad B=C\ [\Gamma ]\qquad c\in C  \ [\Gamma]
\qquad c=d\in C  \ [\Gamma]
 $$
i.e. new types, new elements of types, and new equalities between them.
\end{definition}

\begin{definition}[syntactic category]
For a given ${\cal T}_{au}$-theory $T$,
let ${\cal C}_{T}$ be the syntactic category built out of $T$
as in \cite[section 5.2]{Maietti:ModCor}.
\end{definition}

\begin{definition}[internal theory of an AU as an AU]
Given an arithmetic universe $\mathcal{A}$,
let $T(\mathcal{A})$ be the ${\cal T}_{au}$-theory
that is the  internal language of $\mathcal{A}$.
It is defined by the method exemplified with pretoposes in \cite[section 5.4]{Maietti:ModCor}.

Let us call 
$ {\tt Em}^T: \mathcal{A}\rightarrow T(\mathcal{A})$
the embedding of an object in $\mathcal{A}$ as a proper type
and of  a morphism in $\mathcal{A}$ as a proper term
in its internal type theory. Then, let us simply call
$ {\tt Em}: \mathcal{A}\rightarrow {\cal C}_{T(\mathcal{A})}$ 
the embedding of an object $X$ and a morphism $f$ to their copy in   
the syntactic category $ {\cal C}_{T(\mathcal{A})}$
%$\widehat{!^X}$ and $\overline{f}$ 
defined on 
page 1119 of \cite{Maietti:ModCor}.
Finally, let us call ${\tt V}: {\cal C}_{T(\mathcal{A})}\rightarrow \mathcal{A}$
the functor
establishing an equivalence with $ {\tt Em}$ (this called $\epsilon^{-1}_{\cal A}$
in \cite{Maietti:intpre}).

%% Formally  ${\tt V}$ can be defined as follows:
%% it sends a closed type preinterpreted by a fibred
%% functor $\alpha$ to the domain of $\alpha(id): A_{\Sigma}\rightarrow 1$;
%% for given closed types  preinterpreted as $\alpha$ and $\beta$,
%% ${\tt V}$ sends
%% a term  interpreted  as a section $b: A_\Sigma\rightarrow
%% A\times B_{\Sigma}$ of $\beta(\alpha(id))$ in  $\mathcal{A}/A_\Sigma $
%% to the morphism $q(\alpha(id),\beta(id))\cdot B$ (obtained by composing
%% the second projection  $q(\alpha(id),\beta(id))$ of the chosen pullback of 
%% $\beta(id)$ along $ \alpha(id)$ with $b$).
\end{definition}

\begin{definition}[theory defining the free AU]
Given an AU $\mathcal{A}$, let $T_{cat}(\mathcal{A})$ be the free \emph{$\mathcal{T}_{au}$-theory} generated from $\mathcal{A}$ as a category, i.e.
the extension of the typed calculus
${\cal T}_{au}$ with 
the axioms arising from $\mathcal{A}$ considered as a category
according to  definition 5.30 of \cite{Maietti:ModCor}.

Its syntactic category  ${\cal C}_{T_{cat}(\mathcal{A})}$ 
is the free AU generated from $\mathcal{A}$ as a category,
as shown in \cite[section 5.5]{Maietti:ModCor}.
\end{definition}

\begin{definition}
Given an AU $\mathcal{A}$, let ${\cal Y}: \mathcal{A}\rightarrow {\cal C}_{T_{cat}(\mathcal{A})}$ be the functor embedding of theorem 5.31 in \cite{Maietti:ModCor},
sending  an object $X$ and a morphism $f$ to their copy in
$ {\cal C}_{T_{cat}(\mathcal{A})}$. For easiness we keep the same notation here.

Then, let $ {\tt Tr}_{\cal A}: T_{cat}(\mathcal{A})\longrightarrow  T(\mathcal{A})$
be the interpretation functor defined as follows:
it sends proper types  and terms arising respectively from objects
and morphisms of  $\cal A$ to the corresponding ones in  $T(\mathcal{A})$
%(preinterpreted as $\widehat{!^B}$),
%% (interpreted as $\overline{<id,_1 b>}$,
%% where $<id,_1 b>$ is the pullback morphism of $!^B$ along
%% $!^A$ with the notation in \cite{Maietti:ModCor}), 
 and types and terms constructors of $\mathcal{T}_{au}$
to their copy in $T(\mathcal{A})$ according to the interpretation
exemplified for pretopoi in section 5 of  \cite{Maietti:ModCor}.

Then, the functor ${\cal C}({\tt Tr}_{\cal A}): {\cal C}_{T_{cat}(\mathcal{A})}\rightarrow
 {\cal C}_{T(\mathcal{A})}$ sends each closed type and term in $ {\cal C}_{T_{cat}(\mathcal{A})}$  to their translation via ${\tt Tr}_{\cal A}$ in $T(\mathcal{A})$.
\end{definition}

Now we intend to define the internal type theory of an AU $\mathcal{A}$
as the extension of $T_{cat}(\mathcal{A})$ with {\it coherent isomorphisms}
connecting the free AU-structure with the chosen AU-structure in $\mathcal{A}$.
We will call such an internal type theory $T_{iso}(\mathcal{A})$.

Categorically this means that we require the existence
of a natural isomorphism between the identity functor
${\tt Id}: {\cal C}_{T_{cat}(\mathcal{A})} \rightarrow {\cal C}_{T_{cat}(\mathcal{A})}$
and the functor
$$(\, {\cal Y}\cdot {\tt V}\, )\cdot {\cal C}({\tt Tr}_{\cal A}): {\cal C}_{T_{cat}(\mathcal{A})}\rightarrow  {\cal C}_{T(\mathcal{A})}\rightarrow \mathcal {A}\rightarrow {\cal C}_{T_{cat}(\mathcal{A})}$$
Given the importance of this functor we give it a new name:
%% Given that  coherent isomorphisms are indexed over objects of 
%% ${\cal C}_{T_{cat}(\mathcal{A})}$ that are closed types of $T_{cat}(\mathcal{A})$
%% defined by induction, we then
%%  define them
%% by induction on their formation in $T_{cat}(\mathcal{A})$.
\begin{definition}[${\cal A}$-reflection]
Let the functor 
$${\tt R}^{\cal A}:{\cal C}_{T_{cat}(\mathcal{A})}
\rightarrow {\cal C}_{T_{cat}(\mathcal{A})}$$
be defined as ${\tt R}^{\cal A}\, \equiv\,
(\, {\cal Y}\cdot {\tt V}\, )\cdot {\cal C}({\tt Tr}_{\cal A})$
and called the {\it ${\cal A}$-reflector} functor.
\end{definition}

Note also that the {\it ${\cal A}$-reflector} functor
restricted to  $\mathcal{A}$ is essentially the identity:
\begin{lemma}
\label{rest}
For any given AU $\mathcal{A}$
the functor ${\cal Y}: \mathcal{A}\rightarrow {\cal C}_{T_{cat}(\mathcal{A})}$
is naturally isomorphic to ${\tt R}^{\cal A}\cdot {\cal Y}$, that is
the restriction of the {\it ${\cal A}$-reflector} functor on $\mathcal{A}$.
\end{lemma}
\begin{proof}
This follows from the fact that proper types
and terms via $Tr_{\cal A}$ are interpreted in objects and terms
isomorphic to the interpreted ones.
Indeed, for a given object $X$ in $\mathcal{A}$
then $(\, (\, {\cal Y}\cdot {\tt V}\, )\cdot {\cal C}(Tr_{\cal A})\, )
 (\,  {\tt Y}(X)\, )$
is ${\tt V}(X^{\tt Em})$, which is only isomorphic
to ${\cal Y}(X)\, \equiv\, X$ (indeed
$ {\cal C}_{T(\mathcal{A})}$ is only equivalent to $\mathcal{A}$
and not isomorphic to it!).
\end{proof}

%% \begin{definition}[evaluation]
%% We define the following function
%%  $$(-)^{\tt ev}: types (\, T(\mathcal{A})\, )\rightarrow types(\, T_{cat}(\mathcal{A}) \, )$$
%% assigning to a dependent type $ B\  [x_{1}\in C_{1},...,x_{n}\in C_{n}]$
%% preinterpreted by $\gamma_{1}\, ,\, \gamma_{2}\, ,...,\, \gamma_{n}\, ,\,
%% \beta$
%% %\vec{\sigma}(id)\ ,\ \beta(id)$
%% the type preinterpred by 
%% $$\widehat{!^{dom(\gamma_{1}(id))}},\ \widehat{\gamma_{2}}[p_{1}],\ \widehat{\gamma_{3}}[p_{2}],\ ...,
%% \widehat{\gamma_{n}}[p_{n-1}], \ \widehat{t}[p_{n}]$$
%% where $p_{ i}$ is the second projection of the pullback
%% of $\gamma_{i}(id)$ and $p_{i-1}$  for $i=2,\dots, n$.
%% \end{definition}

In order to define  $T_{iso}(\mathcal{A})$ in an explicit way
we need to define the natural isomorphism as a family of isomorphisms
indexed on the objects of ${\cal C}_{T_{cat}(\mathcal{A})}$.
 Since such objects  
are closed types in $T_{cat}(\mathcal{A})$
 that are defined inductively 
out of the whole collection of types in $T_{cat}(\mathcal{A})$,
we thought of describing the desired
natural isomorphism as  a consequence of an isomorphism between
suitable interpretations of $T_{cat}(\mathcal{A})$ in ${\cal C}_{T_{cat}(\mathcal{A})}$. This means that we will  defined a family of suitable isomorphisms
indexed on the whole types  of $T_{cat}(\mathcal{A})$. These isomorphisms
will be called {\it coherent isomorphisms}.

Before proceeding we review some key aspects of how to interpret
 a dependent typed calculus, like  ${\cal T}_{au}$,
 into a category $\cal C$
as defined in  \cite{Maietti:ModCor}.
In particular we review how types, terms with their equalities are interpreted 
together with the
interpretation of substitution and weakening in types, in order to fix the notation of morphisms that will be 
involved in the notion of morphism between interpretations.

First of all the interpretation of a typed calculus
in a category $\mathcal A$ according to \cite{Maietti:ModCor}
is {\it actually} given in the category  $Pgr({\cal A})$ defined as follows:
\begin{definition}
Given a category ${\cal A}$ with terminal object $1$,
the objects of the category {\boldmath $Pgr({\cal A})$}
 are
 finite sequences $b_{1}, b_{2},...,b_{n}$ of morphisms of ${\cal A}$ 
$$
\def\objectstyle{\scriptstyle}
\def\labelstyle{\scriptstyle}
{\xymatrix @-0.5pc{
1&B_{1}\ar[l]^{b_{1}} & B_{2}
\ar[l]^{b_{2}} & \ar@{.}[l] &B_{n}\ar[l]^{b_{n}}  }}
$$
 and a  morphism  from $b_{1}, b_{2},...,b_{n}$
to $c_{1}, c_{2},...,c_{m}$
is a morphism $d$ of   ${\cal A}$   such that $c_{n}\cdot d=b_{n}$ in  ${\cal A}$ 
$
\def\objectstyle{\scriptstyle}
\def\labelstyle{\scriptstyle}
{\xymatrix @-1pc{
&&B_{n} \ar[rr]^{d} \ar[dr]^{b_{n}}
&          & C_{n} \ar[dl]^{b_{n}} \\
  1&   B_{1}\ar[l]^{!_{B_{1}}}&  \ar@{.}[l]    & B_{n-1} \ar[l]^{b_{n-1}} &
}}
$
provided that $n=m$ and $b_{i}=c_{i}$ for $i=1,...,n-1$.
Equality, composition and identity is that induced from $\cal A$.
\end{definition}

Now, given an arithmetic universe $\cal A$,
the interpretation of a dependent type
 $B\ [x\in C_1,..,x_n\in C_n]$
is given by an object in of $Pgr({\cal A})$
$$
\def\objectstyle{\scriptstyle}
\def\labelstyle{\scriptstyle}
{
\xymatrix @-1.5pc{
& & & &&& B_{\Sigma } \ar[dl]^{B^I} \\
  1& &  C_{1 \Sigma }\ar[ll]^{C_1^I}&  \ar@{.}[l] &   & C_{n \Sigma } 
  \ar[ll]^{C_n^I} &
}}
$$
The interpretation of a term  judgement $b\in B\ [\Gamma]$ is 
 a section in ${\cal A}$ of the last morphism $B^I$  of the sequence interpreting the 
dependent type $B$ under the context
$$
\def\objectstyle{\scriptstyle}
\def\labelstyle{\scriptstyle}
{
\xymatrix @-1pc{
& & C_{n\Sigma }\ar[rr]^{b^I} \ar[dr]^{id}&
         & B_{\Sigma } \ar[dl]^{B^I} \\
  1&   C_{1 \Sigma }\ar[l]^{C_1^I}&  \ar@{.}[l]    & C_{n\Sigma } 
  \ar[l]^{C_n^I} &
}}
$$
The equality between types under context is interpreted as equality of
the objects interpreting them in $Pgr({\cal A})$.
The equality between typed terms under context is interpreted as equality between
the sections interpreting them in $Pgr({\cal A})$.

Now we pass to show how substitution of terms  in types 
and weakening of assumptions in types are interpreted in 
 $Pgr({\cal A})$.

The notion of interpretation requires to be able 
  to interpret substitution and weakening   
as follows.
Given  a dependent type $B(x_1,x_2)\ [x_1\in C_1, x_2\in C_2]$ and a term
$c_2\in C_2\ [x_1\in C_1]$ interpreted as
$$
\def\objectstyle{\scriptstyle}
\def\labelstyle{\scriptstyle}
{
\xymatrix @-0.5pc{
&& &B_{\Sigma}\ar[dl]^{B^{I}}\\
1&C_{\Sigma 1}\ar[l]^{C_1^I}&C_{\Sigma 2}\ar[l]^{C_2^I}&\\
}}
\qquad
\def\objectstyle{\scriptstyle}
\def\labelstyle{\scriptstyle}
{
\xymatrix @-0.5pc{
  &C_{\Sigma  1}\ar[rr]^{c_2^I} \ar[dr]^{id}&
         & C_{2 \Sigma }  \ar[dl]^{C_2^I} \\
  1& &  C_{1 \Sigma }\ar[ll]^{C_1^I}&
}}
$$
we interpret
$B(x_1,x_2)[x_2/c_2]\equiv B(x_1,c_2)\ [x_1\in C_1]$
as
$$
\def\objectstyle{\scriptstyle}
\def\labelstyle{\scriptstyle}
{
\xymatrix @-0.5pc{
& &&B[x_2/c_2]_{\Sigma}\ar[dl]^{B[x_2/c_2]^{I}}\\
1&&C_{1 \Sigma }\ar[ll]^{C_1^I}&\\
}}
$$
where the last morphism
 $B[x_2/c_2]^{I}$ is the first projection of the following
substitution diagram:
$$
\def\objectstyle{\scriptstyle}
\def\labelstyle{\scriptstyle}
{\xymatrix @+0.5pc{
B[x_/c_2]_{\Sigma}\ar[rr]^{q_{B[x/c_2]}^I} \ar[d]_{B[x_2/c_2]^{I}}
& & B_{\Sigma} \ar[d]^{B^{I}}      \\
C_{1 \Sigma }\ar[rr]_{c_2^I}&&C_{2 \Sigma } 
}}
$$

Moreover, the type $B\ [x_1\in C_1,y\in D]$
obtained by
weakening 
  the dependent type $B\ [x_1\in C_1]$ with the type
$D\ [x_1\in C_1]$ interpreted as
$$
\def\objectstyle{\scriptstyle}
\def\labelstyle{\scriptstyle}
{
\xymatrix @-0.5pc{
&&  B_{\Sigma}\ar[dl]^{B^{I}}\\
1&C_{1\Sigma}\ar[l]^{C_1^I} &\\
}}
\qquad \def\objectstyle{\scriptstyle}
\def\labelstyle{\scriptstyle}
{
\xymatrix @-0.5pc{
&&  D_{\Sigma}\ar[dl]^{D^{I}}\\
1&C_{1\Sigma}\ar[l]^{C_1^I} &\\
}}
$$
is interpreted as
$$
\def\objectstyle{\scriptstyle}
\def\labelstyle{\scriptstyle}
{
\xymatrix @+0.5pc{
&& & w(B,D)_{\Sigma}\ar[dl]^{w(B,D)^I}\\
1& C_{1 \Sigma }\ar[l]^{C_1^I}& D_{\Sigma}\ar[l]^{D^I}&
%%  & D\times
%% C_{\Sigma 2}\ar[l]^{\pi_1^{D^I\times C_2^I}}&\\
}}
$$
where the last morphism $w(B,D)^I$ is the first projection of
the following weakening diagram:
$$
\def\objectstyle{\scriptstyle}
\def\labelstyle{\scriptstyle}
{\xymatrix @+0.5pc{
w(B,D)^I_{\Sigma}\ar[rrr]^{q_{w(B,D)}^I} \ar[d]_{w(B,D)^I}
&&&  B_{\Sigma} \ar[d]^{B^{I}}      \\
D_{\Sigma }\ar[rrr]_{D^I}&&&C_{1 \Sigma }  
}}
$$
where if the substitution or weakening is performed
 in the middle of the context we still use the same notation as follows.

\noindent
Recall that 
$\Gamma_{j+1}^n\,  \equiv \, x_{j+1}\in 
C_{j+1}, ..., x_{n}\in 
C_{n} $ for a given  context $\Gamma_n $  denoting
with $\Gamma_{o}$ the empty context, then
 the type judgement
$$B [x_j/c_{j}] \ [\Gamma_{j-1}, \Gamma_{j+1}^{n'} ]$$
obtained by substitution with $c_j\in C_J\ [\Gamma_{j-1}]$
where $\Gamma_{j+1}^{n'} \,  \equiv \,
x'_{j+1}\in C_{j+1}[x_j/c_{j}], ..., x'_{n}\in C_{n}[x_j/c_{j}]$
is interpreted as a morphism of $Pgr(\mathcal{A})$
$$
\def\objectstyle{\scriptstyle}
\def\labelstyle{\scriptstyle}
{
\xymatrix @-0.5pc{
& &&&& &&B[x_j/c_j]_{\Sigma}\ar[dl]^{B[x_j/c_j]^{I}}\\
1&&C_{1 \Sigma }\ar[ll]^{C_1^I}&& \ar@{.}[ll]& &C_{n}[x_j/c_{j}]_{\Sigma}\ar[lll]^{C_{n}[x_j/c_{j}]^I}&\\
}}
$$
where, if $\Gamma_{j+1->n}$ is not empty,
 the last morphism is the first projection of the substitution diagram:
$$
\def\objectstyle{\scriptstyle}
\def\labelstyle{\scriptstyle}
{\xymatrix @-0.5pc{
B[x_j/c_j]_{\Sigma}\ar[rrr]^{{q_{B[x_j/c_j]}}^I} \ar[d]_{ B[x_j/c_j]^{I}}
&&&  B_{\Sigma} \ar[d]^{B^{I}}      \\
C_n[x_j/c_j]_{\Sigma }\ar[rrr]_{{q_{C_n[x_j/c_j]}}^I}& && C_{n \Sigma} 
}}
$$
Moreover the type judgement
$ \ B\ [\Gamma_{j},y \in D,\Gamma_{j+1}^n ]\ $
obtained by weakening with a variable in the middle of the context
is interpreted as
$$
\def\objectstyle{\scriptstyle}
\def\labelstyle{\scriptstyle}
{
\xymatrix @-0.5pc{
& &&&& && w(B, D)_{\Sigma}\ar[dl]^{w(B, D)^{I}}\\
1&&C_{1 \Sigma }\ar[ll]^{C_1^I}& & \ar@{.}[ll] & & w(C_{n}, D)_{\Sigma}\ar[lll]^{ w(C_{n}, D)^I}&\\
}}
$$
If $\Gamma_{j+1}^n$ is not empty, then its last morphism
 is the first projection of the following weakening diagram:
$$
\def\objectstyle{\scriptstyle}
\def\labelstyle{\scriptstyle}
{\xymatrix @+0.5pc{
w(B,D)_{\Sigma}\ar[rr]^{{q_{w(B,D)}}^I} \ar[d]_{w(B,D)^I}
& &  B_{\Sigma} \ar[d]^{B^{I}}      \\
w(C_n,D)_{\Sigma}\ar[rr]_{{q_{w(C_n,D)}}^I} &&C_{n \Sigma } 
}}
$$
Then we give the following definition of {\it generic interpretation}:
\begin{definition}
A {\em  generic} interpretation
of a typed calculus $T$ is one that validates all judgements of the typed
calculus $T$ according to the above notion of judgement interpretations
including substitution and weakening.
\end{definition}
Note that to interpret substitution of terms in types correctly
we need a functorial choice of the above substitution diagrams in $\cal A$.

If we require the substitution and weakening diagrams to be {\it pullbacks} in $\cal A$,
as done in \cite{Maietti:ModCor},
then we need to provide a functorial choice of pullbacks in $\cal A$.

In order to build an interpretation of the type calculus ${\cal T}_{au}$
  in an arithmetic universe
with an arbitrary fixed choice of its structure (and hence with a choice
of pullbacks that is not necessarily functorial), one possibility is
to   define it via
 a preinterpretation
of types and terms into fibred functors
and natural transformations as described in
 \cite{Maietti:ModCor}. Here we refer to this  interpretation defined via fibred functors as a {\it canonical interpretation} of the typed calculus  ${\cal T}_{au}$.
We do not recall the definition of such an interpretation here and we refer the reader to \cite{Maietti:ModCor}.
We just remind that this canonical interpretation is crucial to describe
the internal type theory of an arithmetic universe.

Here we will mention two interpretations of the  typed calculus $T_{cat}(\mathcal{A})$ in ${\cal C}_{T_{cat}(\mathcal{A})}$
that are not defined via fibred functors (hence they are not canonical)
 and are only {\it generic} ones.
These are those interpretations whose action  on the syntactic
category ${\cal C}_{T_{cat}(\mathcal{A})}$  gives rises respectively
to  the identity functor and the  ${\cal A}$-reflector.
This implies that to meet our purpose of building a natural isomorphism
between the identity functor and the  ${\cal A}$-reflector is enough
to build an isomorphism  between the corresponding interpretations.

We now pass to describe the interpretation corresponding 
to the identity functor on ${\cal C}_{T_{cat}(\mathcal{A})}$:
\begin{definition}[The $(-)^{\cal H}$ interpretation]
Let us call $(-)^{\cal H}$
the interpretation of $T_{cat}(\mathcal{A})$ into the
category $Pgr({\cal C}_{T_{cat}(\mathcal{A})})$
 via indexed sums as defined
on page 1138 in \cite{Maietti:ModCor} with the warning of interpreting
the closed type $X$ as $X\rightarrow \top$
 (and not as $\def\objectstyle{\scriptstyle}
 \def\labelstyle{\scriptstyle}
 {\xymatrix @-0.5pc{ \Sigma_{z\in \top} X\ar[r]^{\pi_1}&\top}}$).
 For example
a type $B(x)\ [x\in C]$ is interpreted as 
$$\def\objectstyle{\scriptstyle}
 \def\labelstyle{\scriptstyle}
 {\xymatrix @-0.5pc{ 
\Sigma_{z\in  C} B(z)\ar[rr]^{\ {\pi}_1 } &&
C \ar[rr]^{\ !^C }& & \top}}$$
and a term $b(x)\in B(x)\ [x\in C]$ is interpreted as a section $<id, b(x)>$
 of the interpretation $\pi_1$ of its type, namely 
$\pi_1\cdot <id, b(x)>=id$
in   ${\cal C}_{T_{cat}(\mathcal{A})}$.
\end{definition}
In essence 
$(-)^{\cal H}$  interprets types and terms in themselves as indexed sum types
and sections. Hence it 
corresponds  on ${\cal C}_{T_{cat}(\mathcal{A})}$ to the identity functor.

Then, the interpretation of $T_{cat}(\mathcal{A})$ in $Pgr({\cal C}_{T_{cat}(\mathcal{A})})$ corresponding to
the  ${\cal A}$-reflector functor
is obtained by turning  the translation
 ${\tt Tr}_{\cal A}:T_{cat}(\mathcal{A})
\longrightarrow  T(\mathcal{A})$, used to built the reflector,
 into an interpretation
$Int_{\mathcal A} :T_{cat}(\mathcal{A})
\longmapsto Pgr(\mathcal{A}) $
 by composing ${\tt Tr}_{\cal A}$ 
with the semantic
denotation of $T(\mathcal{A})$-types and terms.
Then, by using the embedding functor  ${\cal Y}: \mathcal{A}\rightarrow  {\cal C}_{T_{cat}(\mathcal{A})}$ we can think of $Int_{\mathcal A}$ in 
$Pgr({\cal C}_{T_{cat}(\mathcal{A})})$ by keeping the same name
$$Int_{\mathcal A} :T_{cat}(\mathcal{A})
\longmapsto Pgr({\cal C}_{T_{cat}(\mathcal{A})}) $$
Note that this interpretation induces the   ${\cal A}$-reflector functor
on ${\cal C}_{T_{cat}(\mathcal{A})}$.

Now our task is to build an isomorphism between
the interpretations  $(-)^{\cal H}$  and 
$Int_{\mathcal A}$. In order to do so we
need to first define the notion of morphism between interpretations.
To this purpose we
 can not work in 
$Pgr({\cal C}_{T_{cat}(\mathcal{A})}) $ but we  pass to consider
 the category of arrow lists ${\cal C}_{T_{cat}(\mathcal{A})}^{\rightarrow fin}$ defined in \cite{Maietti:ModCor} as follows:

\begin{definition}[Category of arrow lists]
Given a category $\cal A$,
we define the category ${\cal A}^{\rightarrow fin}$ 
as follows:
 its objects are sequences 
$$
\def\objectstyle{\scriptstyle}
\def\labelstyle{\scriptstyle}
{\xymatrix @-0.5pc{
C_{0}&C_{1}\ar[l]^{c_{1}} & C_{2}
\ar[l]^{c_{2}} & \ar@{.}[l] &C_{n}\ar[l]^{c_{n}}  }}
$$
of composable ${\cal A}$-morphisms 
 and a morphism
 from $c_{1}\ ,\ c_{2}\ ,...,\ c_{n}$  to
$b_{1}\ ,\ b_{2}\ ,...,\ b_{n}$ 
 is a sequence 
 $\phi_{0}\ ,\ \phi_{1}\ ,...,\  \phi_{n}$ of ${\cal A}$-morphisms 
 such that all the following squares commute in $\cal A$
 $$
\def\objectstyle{\scriptstyle}
\def\labelstyle{\scriptstyle}
{\xymatrix @-1pc{
C_{n} \ar[r]^{\phi_{ n}} \ar[d]_{c_{n}}
&           B_{n} \ar[d]^{b_{n}} \\
   C_{n-1}\ar[d]_{c_{n-1}}  \ar[r]^{\phi_{ n-1}}   
 & B_{n-1} \ar[d]^{b_{n-1}}   \\
 \ar@{.}[d] & \ar@{.}[d]  \\
 C_{2}\ar[d]_{c_{2}}\ar[r]^{\phi_{ 2}}&  B_{2}\ar[d]^{b_{2}}\\
  C_{1}\ar[d]_{c_{1}}\ar[r]^{\phi_{ 1}} & B_{1}\ar[d]^{b_{1}}\\
C_{0}\ar[r]^{\phi_{ 0}}  &B_0
}}
$$
Two morphisms are equal if their $n$-th components are equal
for each $n$ in the list.
Composition of morphisms  is a morphism whose
$n$-th component is the composition of the $n$-th components 
of the given morphisms, and the identity is the morphism
whose components are all identities.
 
%% The  objects of  ${\cal A}_{T_(\mathcal{A})}^{{\rightarrow}_{fin}}$   
%% are the dependent types 
%% $B(x)\ [\Gamma]$.
%% A morphism of  $T^{{\rightarrow}_{fin}}$ from the dependent type
%% $B(x)\ [\Gamma]$ to $ B'\  [x'_{1}\in C'_{1},...,x'_{n}\in C'_{n}]$, 
%% is given by a series of terms of $T$
%% $$ b'\in B'(c'_{1},...,c'_{n})\ [\Gamma,y\in B]
%% $$ 
%% where $c_{1}\in C'_{1}\ [\Gamma, y\in B]$
%% and  $c'_{i}\in C'_{i}(c'_{1},...,c'_{i-1})\ [\Gamma, y\in B]$
%% for $ i=1,...,n$ are derivable terms (see also \cite{Mai2}).

%% We may represent a morphism in $T^{{\rightarrow}_{fin}}$ 
%% as the following picture
%% $$
%% \def\objectstyle{\scriptstyle}
%% \def\labelstyle{\scriptstyle}
%% {\xymatrix @-1pc{ B \ar[r]^{\phi_{ B}} \ar[d]_{}
%% &           B' \ar[d]^{} \\
%%    C_{n}\ar[d]_{}  \ar[r]^{\phi_{ C_{n}} }  
%%  & C'_{n} \ar[d]^{ }   \\
%%  \ar@{.}[d] & \ar@{.}[d]  \\
%%  C{2}\ar[d]_{a_{2}}\ar[r]^{\phi_{C_2}}&  C'_{2}\ar[d]^{}\\
%%   C_{1}\ar[r]^{\phi_{C_1}} & C'_{1}
%% }}
%% $$

\end{definition}

Observe that  $Pgr({\cal A})$
 is a subcategory of ${\cal A}^{{\rightarrow}_{fin}}$ with the same objects
and where morphisms are all identities except for the last component.

Then we define the notion of {\it interpretation morphism}, 
and the associated one of  {\it interpretation isomorphism}, between
interpretations of a generic  ${\cal T}_{au}$-theory $T$ into a generic arithmetic
universe $\cal B$, or better in $Pgr(\mathcal{B})$, by relating them
in $ {\cal B}^{\rightarrow fin}$ as follows.
\begin{definition}\label{isointe}
Given a ${\cal T}_{au}$-theory $T$ and an arithmetic universe $\cal B$ and two
interpretation  $Int_1: T\longmapsto Pgr(\mathcal{B})$
and $Int_2: T\longmapsto Pgr(\mathcal{B})$
of $T$, we say that there is an {\em  morphism of 
interpretation} from $Int_1$ to $Int_2$
$$\sigma(-): Int_1\longrightarrow Int_2$$
if for each type judgement $B\ [\Gamma]$ of $T$
there exists an morphism in
 $\mathcal{B}^{{\rightarrow}_{fin}}$
$$\sigma_{B\ [\Gamma]}: (\, B\ [\Gamma]\, )^{Int_1}\rightarrow
 (\, B\ [\Gamma]\, )^{Int_2}$$
%%  i.e. it has an inverse $\sigma_{B\ [\Gamma]}^{-1}$ in 
%% $\mathcal{B}^{{\rightarrow}_{fin}}$
%% i.e.  such that
%% $$\sigma_{B\ [\Gamma]}\cdot \sigma_{B\ [\Gamma]}^{-1}=id\qquad 
%% \sigma_{B\ [\Gamma]^{-1}}\cdot \sigma_{B\ [\Gamma]}=id$$

Moreover, supposed to represent
  $$\begin{array}{l}
(\, B\ [\Gamma]\, )^{Int_1}\, \equiv\, 
\def\objectstyle{\scriptstyle}
\def\labelstyle{\scriptstyle}
{\xymatrix @-0.5pc{ {B_{\Sigma}}^{Int_1} \ar[rr]^{ B^{Int_1}}
& & {C_{n\Sigma}}^{Int_1}\ar[rr]^{C_{n}^{Int_1}} & &
\ar@{.}[r] &  {C_{1 \Sigma }}^{Int_1}\ar[rr]^{ C_{1}^{Int_1}} && 1^{Int_1}}}\\
(\, B\ [\Gamma]\, )^{Int_2}\, \equiv\, 
\def\objectstyle{\scriptstyle}
\def\labelstyle{\scriptstyle}
{\xymatrix @-0.5pc{ {B_{\Sigma}}^{Int_2} \ar[rr]^{ B^{Int_2}}
&& {C_{n\Sigma}}^{Int_2}\ar[rr]^{C_{n}^{Int_2}} &&
\ar@{.}[r] &  {C_{1 \Sigma }}^{Int_2}\ar[rr]^{ C_{1}^{Int_2}} && 1^{Int_2}}}
\end{array}
$$

we represent $\sigma_{B\ [\Gamma]}$ in $\mathcal{B}^{{\rightarrow}_{fin}}$ as follows:
$$
\def\objectstyle{\scriptstyle}
\def\labelstyle{\scriptstyle}
{\xymatrix @-0.5pc{{ B_{\Sigma}}^{Int_1} \ar[r]^{\sigma_{ B}} \ar[d]_{ B^{Int_1}}
&           {B_{\Sigma}}^{Int_2}  \ar[d]^{B^{Int_2}} \\
   {C_{n\Sigma}}^{Int_1}\ar[d]_{C_{n}^{Int_1}}  \ar[r]^{\sigma_{ C_{n}} }  
 &{ C_{n\Sigma}}^{Int_2} \ar[d]^{C_{n}^{Int_2} }   \\
 \ar@{.}[d] & \ar@{.}[d]  \\
 {C_{2\Sigma }}^{Int_1}\ar[d]_{C{_2}^{Int_1}}\ar[r]^{\sigma_{C_2}}& { C_{2\Sigma}}^{Int_2}\ar[d]^{C_{2}^{Int_2}}\\
  {C_{1 \Sigma }}^{Int_1}\ar[r]^{\sigma_{C_1}}\ar[d]_{ C_{1}^{Int_1}} &
{ C_{1\Sigma}}^{Int_2}\ar[d]^{C_{1}^{Int_2}}\\
1^{Int_1}\ar[r]^{\sigma_{C_1}} & 1^{Int_2}
}}
$$
Then we require that each component $\sigma_{B}$ satisfy  the following conditions:
\begin{list}{-}{ }
\item {\tt naturality condition}
the last component of $\sigma_{B\ [\Gamma]}$ commutes with
the interpretation of its terms: for every term judgement $ b\in B\ [\Gamma]$
of $T$ $$\sigma_{B}\cdot b^{Int_1}= b^{Int_2}\cdot\sigma_{C_{n}}$$
in $ {\cal B}$, supposed  $\Gamma\, \equiv\,
x_1\in C_1,\dots , x_n\in C_n$
\item
{\tt weakening condition}
  the last component of $\sigma_{B\ [\Gamma]}$  commutes with the interpretation of weakening: for every judgement 
 $D\ type\ [\Gamma_{j}]$ in $T$ with $\Gamma_{j}$ sublist of $\Gamma$
 %% -supposed to be  interpreted as
 %% $c_1\ ,\  c_2(id)\ ,...,\ c_I,{j}(id)\ ,\ m$- 
 $$\sigma_{B}\cdot  q_{w(B,D)}^{Int_1} =q_{w(B,D)}^{Int_2}\cdot
\sigma_{w(B,D) }$$
where $w(B,D)$ is the type $B$ weakened on $D$ and
$q_{w(B,D)}^{Int_1}$ is the second projection of the weakening diagram of the last morphism
interpreting $B$ according to $Int_1$  along the context weakened with $D$.
 %% with $\widetilde{M}\, \equiv\,
%% dom( (\,M\ [\Gamma_j]\, )^\Sigma)$
%% and  $p_{j}\  \equiv \ \pi_1^{M}$
%%  and  $p_{i}\  \equiv \ p_{i-1}\times id$ for $i=j+1,..., n$.

 \noindent
 \item {\tt substitution condition}
the last component of $\sigma_{B\ [\Gamma]}$  commutes with the interpretation of substitution: 
for every term judgement $c_{j}\in C_{j}\ [\Gamma_{j-1}]$ in $T$ with $\Gamma_{j-1}$ sublist of $\Gamma$
 $$\sigma_{B}\cdot  q_{B[x/c_j]}^{Int_1}=q_{B[x/c_j]}^{Int_2} \cdot\sigma_{B[x_j/c_{j}]}$$
where $q_{B[x/c_j]}^{Int_1}$ is the second projection of the substitution diagram of the last morphism
interpreting $B$ according to $Int_1$ along the morphism
 expressing the substitution with $c_j$.
\end{list}

The interpretation morphism is an {\em interpretation isomorphism}
if each   
$$\sigma_{B\ [\Gamma]}: (\, B\ [\Gamma]\, )^{Int_1}\rightarrow
 (\, B\ [\Gamma]\, )^{Int_2}$$
is an isomorphism in $\mathcal{B}^{{\rightarrow}_{fin}}$,
 i.e. it has an inverse $\sigma_{B\ [\Gamma]}^{-1}$ in 
$\mathcal{B}^{{\rightarrow}_{fin}}$
i.e.  such that
$$\sigma_{B\ [\Gamma]}\cdot \sigma_{B\ [\Gamma]}^{-1}=id\qquad 
\sigma_{B\ [\Gamma]^{-1}}\cdot \sigma_{B\ [\Gamma]}=id$$

\end{definition}

Now, we are ready to give the definition of $T_{iso}(\mathcal{A})$
as the extension of $T_{cat}(\mathcal{A})$ with a natural isomorphism
between  the interpretations $ (-)^{\cal H}$ and $(-)^{\cal A}$
via {\it coherent isomorphisms}:
\begin{definition}[$T_{iso}(\mathcal{A})$=$T_{cat}(\mathcal{A})$+ coherent isos]\label{coiso}
Given an AU $\mathcal{A}$, let us consider the
above interpretations $(-)^{\cal H}$ and  $(-)^{\cal A}$ of $T_{cat}(\mathcal{A})$
in $Pgr({\cal C}_{T_{cat}(\mathcal{A})})$.

Then
we define $T_{iso}(\mathcal{A})$ as the   \emph{$\mathcal{T}_{au}$-theory} extending $T_{cat}(\mathcal{A})$ with new terms and equalities
formalizing the existence of an  {\it isomorphism of interpretation}
$$\sigma_{-}: (-)^{\cal H}\longrightarrow (-)^{\cal A}$$
in $Pgr({\cal C}_{T_{iso}(\mathcal{A})})$.

Such an isomorphism of interpretation is given by a family of {\it
coherent isomorphisms}
$$\sigma_{B\ [\Gamma]}:
{(B\ [\Gamma])^{\cal H}}_{\Sigma}\longrightarrow {(B\ [\Gamma])^{\cal A}}_{\Sigma}$$
indexed on  any type under context  $B\ [\Gamma]$  of 
$T_{cat}(\mathcal{A})$
satisfying all the  naturality, weakening and substitution conditions
of an isomorphism of interpretation as in definition~\ref{isointe}
with respect to types and terms in $T_{cat}(\mathcal{A})$.
Now we proceed to define such coherent isomorphisms by induction on types
and terms of $T_{cat}(\mathcal{A})$.
\em

In order to define a coherent isomorphism $\sigma_{B\ [\Gamma]}$
indexed on a type $B\ [\Gamma]$ interpreted by a limit
(as the terminal type, the equality type), we actually 
define its inverse
 $\sigma_{B\ [\Gamma]}^{-1}$ as the induced morphism
from the universal property of the limit.
Instead  we define a coherent isomorphism $\sigma_{B\ [\Gamma]}$
indexed on a type $B\ [\Gamma]$ interpreted by a colimit
(as the false type, the sum type, the quotient type)
or by an initial algebra (the list type) 
directly as  the induced morphism
from the universal property of the colimit (or of the initial algebra).

Hence, the coherent isomorphism $\sigma_{\top}$
indexed on the terminal type is defined as the inverse of 
$\sigma_{\top}^{-1}: \top^{\cal A}_{\Sigma}\rightarrow \top $, where $\top^{\cal A}_{\Sigma}$ is a terminal object in $\cal A$,
(that is the domain
interpretation of the terminal type). In turn
$\sigma_{\top}^{-1}$ is defined as the unique morphism in $\mathcal{C}_{ T_{cat}(\mathcal{A})}$ to  the terminal object  $\top$ of  $\mathcal{C}_{ T_{cat}(\mathcal{A})}$.
 
The coherent isomorphism indexed on  $\top\ [\Gamma]$ weakened on a context  is defined
in a way as to satisfy the weakening condition.
%% $(\top\ [\Gamma])^{\cal H}\, \equiv\,w(\top, \Gamma)^{\cal H}\rightarrow C_n^{\cal H}\dots..\top$
 %% we define the coherent isomorphism
%% $$\sigma_{\top[\Gamma]}^-1: w(\top, \Gamma)^{\cal A}\rightarrow  
%% w(\top, \Gamma)^{\cal H}$$
%% as $\sigma_{\top[\Gamma]}^-1\, \equiv\, < \sigma_{\Gamma}^-1\cdot w(\top, \Gamma)^{\call A},
%%  \sigma_{\top}^-1\cdot q(\Gamma,\top)^{\cal A} >$.

The coherent isomorphism indexed on
 any proper type $C$ coming from $\cal A$
$$\sigma_{C}:
{C^{\cal H}}_{\Sigma}\rightarrow {C^{\cal A}}_{\Sigma}$$
is the isomorphism coming from the natural isomorphism of ${\cal V}\cdot
{\tt Em}: {\cal A}\rightarrow  {\cal C}_{T(\mathcal{A})}\rightarrow {\cal A}$
with the identity (recall that  ${\cal V}$ and ${\tt Em}$
 gives an equivalence between ${\cal A}$ and  ${\cal C}_{T(\mathcal{A})}$),
since   ${C^{\cal H}}_{\Sigma}\, \equiv\, C$ while ${C^{\cal A}}_{\Sigma}\, \equiv\, {\cal V}\cdot {\tt Em}(C)$. Note that this isomorphism is  in $\cal A$.

 %% $C^{\cal H}\, \equiv\, C\rightarrow \top$
%% and $(C^{\cal H})^{\cal A}\, \equiv\, \widehat{!^C}(id_{1_{\cal A}})$
%% where $1_{\cal A}$ is the terminal object of $\cal A$ we put
%% $$\sigma_{C}\, \equiv\,  <!^C,_{1_{\cal A}}\  id_C> $$
%% where  
%% $<!^C,_{\top}\  id_C>: C\rightarrow dom (\, \widehat{!^C}(id_{1_{\cal A}})\, )$
%%  is the unique morphism in $\cal A$ 
%% toward the pullback of $(!^C)^{\cal A}$ along the identity $id_{1_{\cal A}}$.
In the next to simplify the notation, given
 a context $\Gamma\, \equiv\, x_1\in C_1,\dots, x_n\in C_n]$,
we simply indicate the component  $\sigma_{C_n}$
of the last context assumption with $\sigma_{\Gamma}$ as the context consisted of one single assumption.

We define the coherent isomorphism indexed on the Indexed Sum type
$$\sigma_{\Sigma_{x\in D}\, B(x) [\Gamma]}: {( \Sigma_{x\in D}\, B(x) )^{\cal H}}_{\Sigma}\rightarrow
{( \Sigma_{x\in D}\, B(x) )^{\cal A}}_{\Sigma}$$
as the inverse of $\sigma_{\Sigma_{x\in D}\,B(x) [\Gamma]}^{-1}$ defined in turn
as follows.
Observe that the last morphism
interpreting $\Sigma_{x\in D}\, B(x)\ [\Gamma]$ according to  $(-)^{\cal A}$
is  $( \Sigma_{x\in D}\, B(x)\ [\Gamma] )^{\cal A}\, \equiv\,
D^{\cal A} \cdot B^{\cal A}$.
Moreover observe that the last
morphism interpreting $\Sigma_{x\in D}\,  B(x)$ according to $(-)^{\cal H}$
is isomorphic to $  B^{\cal H}\cdot D^{\cal H}$ in ${\cal C}_{T_{cat}(\mathcal{A})}/
\Gamma^{\cal H}_{\Sigma}$.
 Hence we define
 $\sigma_{\Sigma_{x\in D}\, B(x) [\Gamma]}^{-1}\, \equiv\, \nu\cdot
\sigma_{B(x)[\Gamma, x\in D]}^{-1}$
where $\nu:B^{\cal H}\cdot D^{\cal H}\rightarrow
 (\, \Sigma_{x\in D}\, B(x)\, )^{\cal H}$ is the isomorphism
between the two object in  ${\cal C}_{T_{cat}(\mathcal{A})}/
\Gamma^{\cal H}_{\Sigma}$ (that is defined  as $<\pi_1\cdot \pi_1, <\pi_2\cdot\pi_1 , \pi_2\cdot \pi_2>$ by using the projections $\pi_1,\pi_2$ of the Indexed Sum type).

%% $$\Sigma_{x\in \Gamma^\Sigma}(\ \Sigma_{x\in D^{\cal H}} B(x)^{\cal H}\ )\rightarrow ( \Sigma_{x\in D}B(x) )^{\cal A}$$
%% as $\sigma_{\Sigma_{x\in D}B(x) [\Gamma]} (w)\, \equiv\, \sigma_{B(x)[\Gamma, x\in A]}(
%% << \pi_1(w), \pi_1(\pi_2(w))>, \pi_2\pi_2(w)>)$
%% for $w\in   \Sigma_{x\in \Gamma^{\cal H}}(\ \Sigma_{x\in D^{\cal H}} B(x)^{\cal H}\ )$

We define the isomorphism indexed on the  Equality type 
 $$\sigma_{{\tt Eq}(C, c,d)[\Gamma]}:
{ {\tt Eq}(C, c,d)^{\cal H}}_{\Sigma}\rightarrow {{\tt Eq}(C, c,d)^{\cal A}}_{\Sigma}$$
as follows.
Recall that
  $ {\tt Eq}(C, c,d)^{\cal A}\, \equiv\, 
eq( c^{\cal A}, d^{\cal A})$ is the equalizer of $ c^{\cal A}$ and  $ d^{\cal A}$
in $\cal A$,  as well as ${\tt Eq}(C, c,d)^{\cal H}$ is an equalizer
of $ c^{\cal H}$ and  $ d^{\cal H}$ in  $ {\cal C}_{T_{cat}(\mathcal{A})}$.
Hence we define $\sigma_{{\tt Eq}(C, c,d)[\Gamma]}^{-1}$
 as the unique morphism
toward the equalizer $ {\tt Eq}(C, c,d)^{\cal H}$
 induced by $\sigma_{\Gamma}^{-1}\cdot
 eq( c^{\cal A}, d^{\cal A})$.
This is well defined since
 by hypothesis and naturality of the coherent
isomorphisms we have (recall that equality of morphisms
in $\cal A$ is preserved in  $ {\cal C}_{T_{cat}(\mathcal{A})}$)
 $$ \begin{array}{l}
c^{\cal H}\cdot
 (\, \sigma_{\Gamma}^{-1}\cdot
 eq( c^{\cal A}, d^{\cal A})\, )= \sigma_{C\ [\Gamma]}^{-1}\cdot (\,  c^{\cal A}\cdot
 eq( c^{\cal A}, d^{\cal A})\, ) = \\
= \sigma_{C\ [\Gamma]}^{-1}\cdot (\, d^{\cal A}\cdot
 eq( c^{\cal A}, d^{\cal A})\, )\, )=d^{\cal H}\cdot
 (\, \sigma_{\Gamma}^{-1}\cdot
 eq( c^{\cal A}, d^{\cal A})\, )\end{array}$$

We define the coherent isomorphism indexed on the empty set $\bot$ 
 $$\sigma_{\bot}:\bot\rightarrow \bot^{\cal A}$$
where $\bot^{\cal A}$ is the name of the initial object in $\cal A$,
as the unique morphism in $\mathcal{C}_{ T_{cat}(\mathcal{A})}$
from $\bot$ to $ \bot^{\cal A}$.

Moreover, we define the coherent isomorphism indexed on the empty set
 weakened on a context in a way as to satisfy the weakening condition.

%% $$\sigma_{\bot[\Gamma]}: \Gamma^{\cal H}\times \bot\rightarrow (\Gamma \times \bot)^{\cal A}$$
%% as $\sigma_{\bot[\Ga%% mma]}\, \equiv\, \sigma_{\Gamma}\times \sigma_{\bot}  $ where 
%% %% $(\bot[\Gamma])^{\cal H}\, \equiv\, \Gamma^\Sigma\times \bot\rightarrow C_n^\Sigma\dots\top$.

 The isomorphisms for the quotient type, disjoint sums and lists are
 defined analogously.

Note that the described isomorphism of interpretation is indeed uniquely   determined
from the isomorphisms
 indexed on proper types
(because of the naturality, weakening, substitution conditions).
\end{definition}
%% \begin{definition}
%% Given an AU $\mathcal{A}$, let $T_{iso}(\mathcal{A})$ be the  \emph{$\mathcal{T}_{au}$-theory} extending $T_{cat}(\mathcal{A})$ with the above  {\it coherent
%% isomorphisms} via new proper terms and term equalities in $T_{cat}(\mathcal{A})$
%% denoted by them.
%% \end{definition}

\begin{definition}
\label{equi}
Let  ${\cal Y}_{iso}:\mathcal{A}\rightarrow \mathcal{C}_{ T_{iso}(\mathcal{A})}$ be the functor defined as 
the embedding of an object $X$ and a morphism $f$ to their copy
as they were in $ \mathcal{C}_{ T_{cat}(\mathcal{A})}$.
\end{definition}

Observe that the embedding functor  ${\cal Y}_{iso}$ preserves the AU structure
up to isomorphisms:
\begin{lemma}
The functor  ${\cal Y}_{iso}:\mathcal{A}\rightarrow \mathcal{C}_{ T_{iso}(\mathcal{A})}$ is an AU functor.
\end{lemma}
\begin{proof}
This follows thanks to the presence of coherent isomorphisms.
\end{proof}

We can  prove that the synctactic category associated
to $ T_{iso}(\mathcal{A})$ is equivalent to $\cal A$.
To this purpose we 
 define a translation of $ T_{iso}(\mathcal{A})$ in  $T(\mathcal{A})$:
\begin{definition}
\label{ref}
Let ${\tt St}: T_{iso}(\mathcal{A})\rightarrow T(\mathcal{A})$
be the functor sending any type and term arising respectively from
objects and morphisms of $\cal A$ 
  to the corresponding one in  $T(\mathcal{A})$
%$\widehat{!^B}$
%to $\overline{<id,b>}$, 
 and sending types and terms constructors of $\mathcal{T}_{au}$
 to their copy in $T(\mathcal{A})$. 
Finally coherent isomorphisms get interpreted
as parts of the natural isomorphism between  ${\cal V}\cdot
{\tt Em}: {\cal A}\rightarrow  {\cal C}_{T(\mathcal{A})}\rightarrow {\cal A}$
and the identity.
Indeed, $(B\ [\Gamma])^{\cal A}= {\cal V}\cdot {\tt Em}(\, (B\ [\Gamma])^{\cal H}\, )$

Let $\mathcal{C}({\tt St}): \mathcal{C}_{ T_{iso}(\mathcal{A})}\rightarrow 
{\cal C}_{T(\mathcal{A})}$ be the syntactic functor induced by ${\tt St}$.
\end{definition}
\begin{lemma}
The functor ${\cal Y}_{iso}:\mathcal{A}\rightarrow \mathcal{C}_{ T_{iso}(\mathcal{A})}$
gives rise to an equivalence of category with
the functor ${\tt V}_{iso}\, \equiv\, {\tt V}\cdot \mathcal{C}({\tt St}): \mathcal{C}_{ T_{iso}(\mathcal{A})}\rightarrow{\cal C}_{T(\mathcal{A})}\rightarrow \mathcal{A}$.
\end{lemma}
\begin{proof}
Clearly $(\, {\tt V}\cdot \mathcal{C}({\tt St})\, )\cdot {\cal Y}_{iso}$
is naturally isomorphic to the identity.
Instead we prove that ${\cal Y}_{iso} \cdot (\, {\tt V}\cdot \mathcal{C}({\tt St})\, )$ is isomorphic to the identity thanks to coherent isomorphisms
when the functor is applied to $\mathcal{T}_{au}$-constructors.
\end{proof}

This means that we can speak of $T_{iso}(\mathcal{A})$ as {\it the internal
theory of $\cal A$ with coherent isomorphisms}.

Now our purpose is to prove that given two arithmetic universes
$\cal A$ and $\cal B$, the AU functors from $  \cal A$ to $\cal B$
correspond to translations between their internal theories
with coherent isomorphisms, i.e. to translations from $T_{iso}(\mathcal{A})$ to
$T_{iso}(\mathcal{B})$.
To this purpose we first lift an AU functor to a translation
between the corresponding
 free theories generated from the arithmetic universes:
\begin{definition}
Given the arithmetic universes ${\cal A}$ and $\cal B$ with 
 an AU functor $F: \mathcal{A}\rightarrow \mathcal{B}$,
we can define a translation between the free  $\mathcal{T}_{au}$-theories generated
from them 
 $$(-)^F:T_{cat}(\mathcal{A})\rightarrow T_{cat}(\mathcal{B}) $$
as follows: $(-)^F$ translates  types and terms arising from $\cal A$ via $F$, i.e. 
 each proper type  arising from an object
$C$ of $\mathcal{A}$
is translated into $F(C)$ and  each proper term arising from a morphism $c$
is translated into $F(c)$;
moreover $\mathcal{T}_{au}$-constructors are interpreted as the corresponding ones in $T_{cat}({\cal B})$.
\end{definition}

\begin{lemma}\label{fsum}
Given an AU functor $F: \mathcal{A}\rightarrow \mathcal{B}$, the
translation $(-)^F:T_{cat}(\mathcal{A})\rightarrow T_{cat}(\mathcal{B}) $
induced between the corresponding free theories satisfies the following:
for any judgement $B\ [\Gamma]$ then
$$ (\, (\, B\ [\Gamma]\, )^{\cal H}\, )^F\, \equiv\, (\, B\ [\Gamma]\, )^{F}\, )^{\cal H}$$
\end{lemma}
\begin{proof}
It follows from the fact that $(-)^F$ is a translation and hence it preserves
indexed sums strictly.
\end{proof}
\begin{lemma}\label{fcor}
Given an AU functor $F: \mathcal{A}\rightarrow \mathcal{B}$, the
translation $(-)^F:T_{cat}(\mathcal{A})\rightarrow T_{cat}(\mathcal{B}) $
induced between the corresponding free theories
allows to define the following interpretations of $T_{cat}(\mathcal{A})$ in $Pgr({\cal B})$
$${(-)^F}^{\cal B}: T_{cat}(\mathcal{A})\longrightarrow Pgr({\cal B})\qquad
{(-)^{\cal A}}^F: T_{cat}(\mathcal{A})\longrightarrow Pgr({\cal B})$$
(by precomposing   $(-)^F$ with $(-)^{\cal B}$ and postcomposing it with
 $(-)^{\cal A}$)
between which there exists an
 isomorphism of interpretation
$$\tau (-):{(- )^{F}}^{\cal B}\longrightarrow 
{ (- )^{\cal A}}^F$$
 \end{lemma}

\begin{proof}
We define the required isomorphism of interpretation by using
the coherent isomorphisms of $F$ needed to preserve
 the AU structure.

For example the isomorphism indexed on the terminal type
$\tau_{\top}: \top^{\cal B}_{\Sigma}\rightarrow F(\top^{\cal A})_{\Sigma}$
is the part of the
coherent isomorphism of $F$ preserving the terminal
object of $\cal A$ represented by $ \top^{\cal A}$
from
the terminal object of $\cal B$ given by $\top^{\cal B}_{\Sigma}$.

Moreover, for any proper type $C$ coming from $\cal A$  we define
$$\tau_{C}:
{F(C)^{\cal B}}_{\Sigma}\rightarrow
F({C^{\cal A}}_{\Sigma})$$
as the composition of the following isomorphisms 
${F(C)^{\cal B}}_{\Sigma}\simeq F(C)\simeq F({C^{\cal A}}_{\Sigma}) $ all derived
from the natural isomorphism of $ {\cal V}\cdot {\tt Em}$ with the identity
both for  ${\cal A}$ and ${\cal B}$:
indeed ${C^{\cal A}}_{\Sigma}\, \equiv\, {\cal V}\cdot {\tt Em}(C)$
is isomorphic to $C$ in ${\cal A}$, hence in ${\cal C}_{T_{cat}(\mathcal{A})}$
which gives an isomorphism $F(C)\simeq F({C^{\cal A}}_{\Sigma})$,
as well as ${F(C)^{\cal B}}_{\Sigma}\, \equiv\, {\cal V}\cdot {\tt Em}(F(C))$
is isomorphic to $F(C)$ for the analogous reason.

The coherent isomorphism indexed on the Indexed Sum type
$$\tau_{\Sigma_{x\in D}\, B(x) [\Gamma]}: {( \Sigma_{x\in D}\, B(x) )^F}^{\cal B}_{\Sigma}\rightarrow
{( \Sigma_{x\in D}\, B(x) )^{\cal A}}^F_{\Sigma}$$
is  defined 
as follows.
Observe that the last morphism
interpreting $( \Sigma_{x\in D}\, B(x) )^F$ according to $(-)^{\cal B}$
is  ${( \Sigma_{x\in D}\, B(x) )^F}^{\cal B}\, \equiv\,
{D^F}^{\cal B} \cdot {B^F}^{\cal B}$.
Moreover for the same reason $(\, \Sigma_{x\in D}\,  B(x)\, )^{\cal A}\, \equiv\,
{D}^{\cal A} \cdot {B}^{\cal A}$ and hence
 ${(\, \Sigma_{x\in D}\,  B(x)\, )^{\cal A}}^F\, \equiv\,
F({D}^{\cal A}) \cdot F({B}^{\cal A})$.
Therefore we define
 $\tau_{\Sigma_{x\in D}\, B(x)\, [\Gamma]}\, \equiv\, 
\tau_{B(x)\, [\Gamma, x\in D]}$.

The coherent isomorphism indexed on the  Equality type 
 $$\tau_{{\tt Eq}(C, c,d)[\Gamma]}: {{\tt Eq}(C, c,d)^F}^{\cal B}_{\Sigma}\rightarrow {{\tt Eq}(C, c,d)^{\cal A}}^F_{\Sigma}$$
is defined as the inverse of  $\tau_{{\tt Eq}(C, c,d)[\Gamma]}^{-1}$
defined in turn  as follows.
Recall that
  ${\tt Eq}(C, c,d)^{\cal A}\, \equiv\, 
eq( c^{\cal A}, d^{\cal A})$ is the equalizer of $ c^{\cal A}$ and  $ d^{\cal A}$
in $\cal A$. Then, by coherent isomorphisms of $F$ preserving the
AU structure we know that $F( {\tt Eq}(C, c,d) ^{\cal A})$ is an equalizer
of $F(c^{\cal A})$ and  $F(d^{\cal A})$ in ${\cal B}$.
Moreover, also ${{\tt Eq}(C, c,d)^F}^{\cal B}\, \equiv\,
{{\tt Eq}(C^F, c^F,d^F)}^{\cal B} $
is an equalizer of $ {c^F}^{\cal B}$ and  $ {d^F}^{\cal B}$ in ${\cal B}$.
Therefore
we define
$\tau_{{\tt Eq}(C, c,d)[\Gamma]}^{-1}$  as the unique morphism
toward the equalizer $ {{\tt Eq}(C^F, c^F,d^F)}^{\cal B} $
 induced by $\tau_{\Gamma}^{-1}\cdot
 F(eq( c^{\cal A}, d^{\cal A}))$.
This is well defined with an argument analogous to that in definition~\ref{coiso}.

The isomorphisms on the other types are defined analogously.
\end{proof}

Now recall from page 1143 of  \cite{Maietti:ModCor}
that we can view a theory as a category and  a translation as a functor.
Hence, given AU's $\mathcal{A}$ and $\mathcal{B}$, we can think
of the collection of translations
from $T_{iso}(\mathcal{A})$ to $T_{iso}(\mathcal{B})$ as a category
 $Th( T_{iso}(\mathcal{A}),
T_{iso}(\mathcal{B}))$ with translations as objects and natural transformations
as morphisms.
Hence we state the following correspondence between AU functors
and translations between internal theories with coherent isomorphisms:
\begin{theorem}
\label{main}
For any AU's $\mathcal{A}$ and $\mathcal{B}$,
there is an equivalence between the category $AU(\mathcal{A}, \mathcal{B})$ 
of AU functors and natural transformations
and the category  $Th( T_{iso}(\mathcal{A}),
T_{iso}(\mathcal{B}))$ of translations and natural transformations.
\end{theorem}
\begin{proof}
Given an AU functor $F: \mathcal{A}\rightarrow \mathcal{B}$
we define the translation $T(F): T_{iso}(\mathcal{A})\rightarrow T_{iso}(\mathcal{B})$ as follows:
 $T(F)$ interprets  types and terms arising from $\cal A$ via $F$, i.e.
 each proper type  arising from an object
$C$ of $\mathcal{A}$
is translated into the specific type of $T_{iso}({\cal B})$
arising from
$F(C)$,
and  each specific term arising from a morphism $c$ of  $\mathcal{A}$
is translated into the term arising from $F(c)$;
moreover $\mathcal{T}_{au}$-constructors are interpreted as the corresponding ones in $T_{iso}({\cal B})$;
lastly  the interpretation of
a  coherent isomorphism $\sigma_{B\ [\Gamma]}^F$
is given as the composition of a suitable coherent isomorphism
of  $T_{iso}({\cal B})$  with $\tau_{B\ [\Gamma]}$ in lemma~\ref{fcor}:
more in detail
 $$\sigma_{B\ [\Gamma]}^F:{(\, B\ [\Gamma]\, )^{\cal H}}^F_{\Sigma}\longrightarrow
 {(\, B\ [\Gamma]\, )^{\cal A}}^F_{\Sigma}$$
gets interpreted as
$$\sigma_{B\ [\Gamma]}^F\, \equiv\, \tau_{B\ [\Gamma]}\cdot \sigma_{{B\ [\Gamma]}^F}$$
where ${B\ [\Gamma]}^F$ is the translation in $T_{cat}({\cal B})$
of the judgement $B\ [\Gamma]$.  Note
that the domain of $ \sigma_{{B\ [\Gamma]}^F}$ 
can be taken to be ${(\, B\ [\Gamma]\, )^{\cal H}}^F_{\Sigma}$
thanks to lemma~\ref{fsum}.

The translation $T(F)$ is uniquely determined by $F$ up to a
 natural isomorphism
because the interpretation of coherent isomorphisms, given that they commute
with terms, substitution and weakening, is uniquely determined
by interpretation of proper types and terms given by $F$.

Conversely any translation $L: T_{iso}(\mathcal{A})\rightarrow T_{iso}(\mathcal{B})$ gives rise to an AU functor ${\cal C}(L): {\cal C}_{T_{iso}(\mathcal{A})}\rightarrow {\cal C}_{T_{iso}(\mathcal{B})}$ defined on objects and morphisms
in  ${\cal C}_{T_{iso}(\mathcal{A})}$ as their translations in
$T_{iso}(\mathcal{B})$.
Finally ${\tt V}_{iso}\cdot (\, {\cal C}(L)\cdot {\cal Y}_{iso}\, ):
{\cal A}\longrightarrow {\cal C}_{T_{iso}(\mathcal{A})}\rightarrow {\cal C}_{T_{iso}(\mathcal{B})}\rightarrow {\cal B}$ gives an AU functor as desired.

The given correspondence establishes an equivalence of categories.
\end{proof}

From this we can deduce the following:
\begin{corollary}
\label{inteq}
Given the AU's $\mathcal{A}$ and $\cal B$,
the category of interpretations of  $T_{iso}(\mathcal{A})$  into $\cal B$
as in section 5  of \cite{Maietti:ModCor} with interpretation morphisms is
in equivalence with the category of AU functors from  $\mathcal{A}$ to 
$\cal B$.
\end{corollary}
\begin{proof}
Giving an interpretation ${\cal J}$ as in section 5 of \cite{Maietti:ModCor} 
means to give a translation $Tr_{\cal J}$ from  $T_{iso}(\mathcal{A})$  to $T({\cal B})$ (because types and terms of $T({\cal B})$ are defined together
with their interpretation in $\cal B$).
Hence, from  \cite{Maietti:ModCor}  we know that $Tr_{\cal J}$ provides
 an AU homomorphism between  the corresponding syntactic categories 
${\cal C}(Tr_{\cal J}): {\cal C}_{T_{iso}(\mathcal{A})}\rightarrow {\cal C}_{T_(\mathcal{B})}$. This composed with the suitable parts
of the equivalence of the syntactic categories, respectively with $\mathcal{A}$ and
$\cal B$,  gives an AU functor
$$(\, {\tt V}\cdot (\, {\cal C}(L)\, )\cdot {\cal Y}_{iso}: {\cal A}\longrightarrow {\cal B}$$

Conversely, given an AU functor $F$, by theorem~\ref{main}
we get a translation $T(F): T_{iso}(\mathcal{A})\longrightarrow T_{iso}(\mathcal{B})$ which composed with the translation $\tt St$ in definition~\ref{ref}
gives a translation ${\tt St}\cdot T(F):  T_{iso}(\mathcal{A})\longrightarrow T({\cal B})$. This translation corresponds to an interpretation of
$ T_{iso}(\mathcal{A})$ in  ${\cal B}$ because types and terms
of $T({\cal B})$  are defined with their interpretation in $\cal B$
(i.e. the translation of types and terms of  $T_{iso}(\mathcal{A})$ in 
$T({\cal B})$ comes by definition with the interpretation of them in
${\cal B}$).
\end{proof}

\begin{definition}
\label{sub}
Given an AU $\mathcal{A}$, let $T_{iso}(\mathcal{A})[S]$ be the \emph{$\mathcal{T}_{au}$-theory} extending the typed calculus
${\cal T}_{au}$ with 
 $T_{iso}(\mathcal{A})$ 
and with   some extra AU axioms  $S$ of the form
$$c \in C\ [x\in B]\qquad c=d\in C\ [x\in B]$$ i.e.
we add  extra morphisms and equalities  based on $\mathcal{A}$.
Then we write $\mathcal{A}[S]_t$ for the syntactic category $\mathcal{C}_{ T_{iso}(\mathcal{A})[S]}$.

We then call $\mathcal{I}:\mathcal{A}\rightarrow \mathcal{A}[S]_t$
the functor embedding an object into its type naming it in $\mathcal{A}[S]_t$
and a morphism into the term naming it in $\mathcal{A}[S]_t$.
\end{definition}

\begin{theorem}\label{AUExtnThm}
Let $\mathcal{A}$ and $S$ be as in the above definition.
Then $\mathcal{A}[S]_t$ is universal with respect to
being equipped with an AU functor $\mathcal{I}:\mathcal{A}\rightarrow \mathcal{A}[S]_t$
and an interpretation of the extra structure in $S$
according to the notion of interpretation of a morphism
in section 5.31 of \cite{Maietti:ModCor}: for any AU $\mathcal{B}$,
the category $\mathbf{AU}(\mathcal{A}[S]_t,\mathcal{B})$ is equivalent to the
category of pairs $(F,\alpha)$ where $F:\mathcal{A}\rightarrow \mathcal{B}$ is an functor and $\alpha$ interprets the structure in $S$ with respect to $F$.
\end{theorem}
\begin{proof}
Given an AU functor  $F:\mathcal{A}\rightarrow \mathcal{B}$
we lift it to an interpretation $L^F$ of  $T_{iso}(\mathcal{A})$ in
$\cal B$ by corollary~\ref{inteq} and we extend it to interpret
$T_{iso}(\mathcal{A})[S]$ by interpreting
the new added structure as assigned.

Then the interpretation  $L^F$ seen as a translation from
 $T_{iso}(\mathcal{A})[S]$ to  $T(\mathcal{B})$
gives rise to
a  functor ${\cal C}(L^F): \mathcal{C}_{ T_{iso}(\mathcal{A})[S]}\rightarrow\mathcal{C}_{ T(\mathcal{B})} $
and one from $\mathcal{C}_{ T_{iso}(\mathcal{A})[S]}$ to $\cal B$
defined as $\widetilde{F}\, \equiv\, {\tt V}\cdot {\cal C}(L^F):
\mathcal{C}_{ T_{iso}(\mathcal{A})[S]}\rightarrow
\mathcal{C}_{ T(\mathcal{B})}\rightarrow {\cal B}$.

Any other functor extending $F$ can be proved to be naturally isomorphic
to   $\widetilde{F}$ by induction on the type in $T(\mathcal{A})[S]_t$
as done in theorem 5.31 of  \cite{Maietti:ModCor} (note that also the interpretation of coherent isomorphisms is determined
by $F$ and the interpretation of $S$).
\end{proof}

Now considering that the universal property defining our subspace
 $\mathcal{A}[S]_t$ is the same as that in \cite{mv:arithind}
we conclude that the two notions are equivalent:
\begin{corollary}
Let $\mathcal{A}$ and $S$ be as in definition~\ref{sub}.
Then $\mathcal{A}[S]_t$ is equivalent the notion of subspace
 $\mathcal{A}[S]$ in \cite{mv:arithind}.
\end{corollary}

\begin{remark}
From \cite{mv:arithind}, we recall that examples of subspaces
of an AU  $\mathcal{A}$
are the following:
the subspace ${\cal A}[c:1\rightarrow U]$, called {\it open},
with the addition of a global element $n:1\rightarrow U$
for an object $U$ in $\cal A$,
 is equivalent to the slice category   ${\cal A}/U$;
 the subspace  ${\cal A}[c: \phi\rightarrow \bot]$, called {\it closed},
with the addition of an element from $\phi$, subobject of the terminal object in $\cal A$, to the interpretation of falsum in $\cal A$, is equivalent
to a suitable  category of sheaves.
\end{remark}

\subsection{Classifying category}
Here we prove that the syntactic category $\mathcal{C}_{ T}$
of a ${T}_{au}$-theory $T$ classifies suitable generic interpretations of $T$
in an arithmetic universe $\cal B$.
\begin{definition}
A {\em  standard interpretation} $\cal J$ of a ${T}_{au}$-theory $T$ in an arithmetic
universe $\cal B$
is a {\it generic interpretation} where the substitution and weakening
diagrams are pullbacks and the induced functor on the syntactic category
$${\cal C}({\cal J}): \mathcal{C}_{ T}\longrightarrow {\cal B}$$ 
is {\em an AU functor}.
We recall that ${\cal C}({\cal J})$ is
defined as follows: on closed types $C$ as $dom({\cal J}(C))$
and on terms $f(x)\in B\ [x\in C]$ as  $q(\, {\cal J}(B) \, ,\,  {\cal J}( C )\, ) \cdot  {\cal J}(\, b\in B\ [x\in C]\, )$.
\end{definition}

\begin{definition}[standard interpretation functor]
\label{class}
 Given an arithmetic universe $\mathcal{A}$ and a  ${T}_{au}$-theory $T$,
there exists a {\em standard interpretation functor} from the category 
of arithmetic universes and AU functors to the category of
small categories {\tt Cat}: 
$$Int_T: AU\longrightarrow {\tt Cat}$$
assigning to an arithmetic universe $\cal B$ the category
of standard interpretations with interpretation morphisms
$Int(T, {\cal B})$, and to an AU functor $F: {\cal A}\rightarrow {\cal B}$
the functor
$$Int_T(F): Int(T, {\cal A})\longrightarrow  Int(T, {\cal B})$$
assigning to a standard interpretation $\cal J$ the interpretation
 $\cal J_F$ obtained as follows: if $\cal J$ interprets a type $B\ [\Gamma]$
as $b_{1}, b_{2},...,b_{n}$ with $b_1: C\rightarrow 1$,
 then $\cal J_F$  interprets the same type
as $!^F(C), F(b_{2}),...,F(b_{n})$; and if $\cal J$ interprets a term $b\in B\ [\Gamma]$ as the section $b^{\cal J}$, then $\cal J_F$  interprets  the same term
as  $F(b^{\cal J})$. 
The pullback and weakening diagrams are the value under $F$ of those
 induced by $\cal J$.
This is a standard interpretation because $F$ is an AU functor.
\end{definition}
We can show that the syntactic category of a theory represents
the interpretation functor $Int_T: AU\longrightarrow {\tt Cat}$:
\begin{theorem}
Given a  ${T}_{au}$-theory $T$, its
 interpretation functor $Int_T: AU\longrightarrow {\tt Cat}$ 
is natural isomorphic to the covariant functor $AU(\, {\cal C}_T\, ,\, -\, )$,
and hence for every AU $\cal A$  the category of standard interpretations
of $T$ in $\cal A$ is isomorphic to that of AU functors and natural transformations
$AU(\,  {\cal C}_T\, ,\, { \cal A}\, )$.
\end{theorem}
\begin{proof}
By definition a standard interpretation $\cal J$ of $T$ in $\cal A$ 
induces an AU functor ${\cal C}({\cal J}): {\cal C}_T\longrightarrow { \cal A} $. 
Conversely given an  AU functor $F : {\cal C}_T\longrightarrow { \cal A} $
we define the interpretation ${\cal I}_F$ of  $T$ in $\cal A$ as  $Int_T(F)(
(-)^{\cal H})$
since  the  $(-)^{\cal H}$ intepretation is indeed standard
in  ${\cal C}_T$.
\end{proof}

\section{Acknowledgements}
This work arose as a type theoretic version of the subspace definition in \cite{mv:arithind}
and hence I thank Steve Vickers very much for very stimulating discussions.
I also acknowledge useful discussions with Pino Rosolini and Giovanni Sambin. 
\bibliographystyle{amsalpha}
\bibliography{biblioau}
\end{document}